\documentclass[11pt]{amsart}
\usepackage{amssymb}
\usepackage{times}
\usepackage{url}
\usepackage{hyperref}
\usepackage{cancel}

\newtheorem{theorem}{Theorem}[section]
\newtheorem{lemma}[theorem]{Lemma}
\newtheorem{proposition}[theorem]{Proposition}

\theoremstyle{definition}
\theoremstyle{remark}
\newtheorem{remark}[theorem]{Remark}

\newcommand{\F}{\mathbb{F}}
\newcommand{\Z}{\mathbb{Z}}
\newcommand{\Q}{\mathbb{Q}}

\newcommand{\C}{\mathbb{C}}
\newcommand{\R}{\mathbb{R}}

\newcommand{\GL}{\mathrm{GL}}
\newcommand{\SL}{\mathrm{SL}}
\newcommand{\PSL}{\mathrm{PSL}}
\newcommand{\Sp}{\mathrm{Sp}}

\newcommand{\cl}{\mathrm{cl}}

\begin{document}

\title[Zariski density and computing 
with $S$-integral groups]{Zariski 
density and computing with 
$S$-integral groups}

\author{A.~S.~Detinko}
\address{Department of Computer Science\\
School of Computing and Engineering\\
University of Huddersfield \\
Huddersfield HD13DH \\
UK}
\email{A.Detinko@hud.ac.uk}

\author{D.~L.~Flannery}
\address{School of Mathematical and
Statistical Sciences\\
University of Galway\\
University Road\\
Galway H91 TK33\\
 Ireland}
\email{dane.flannery@universityofgalway.ie}

\author{A.~Hulpke}
\address{Department of Mathematics\\
Colorado State University\\
Fort Collins\\
 CO 80523-1874\\
USA}
\email{Alexander.Hulpke@colostate.edu}


\begin{abstract}
We generalize our methodology
for computing with 
Zariski dense subgroups of 
$\SL(n, \Z)$ and $\Sp(n, \Z)$, 
to accommodate 
input dense subgroups $H$ of 
$\SL(n, \Q)$ and $\Sp(n, \Q)$.
A key task, backgrounded by the 
Strong Approximation theorem, is
computing a minimal congruence 
overgroup of $H$. Once we have this 
overgroup, we may describe 
all congruence quotients of $H$.
The case $n=2$ receives particular 
attention.
\end{abstract}

\maketitle

\section{Introduction}

In \cite{ExpMath,SAT,Density},
we developed a methodology for
computing with finitely generated
Zariski dense subgroups 
 of $\SL(n, \Z)$
and $\Sp(n, \Z)$. 
Our goal now is to generalize
this earlier work to handle dense   
subgroups of $\SL(n, \Q)$ 
 and $\Sp(n, \Q)$.
We are motivated by 
 applications that 
involve representations over $\Q$ 
rather than $\Z$. For example,
 when computing with 
$G\leq \Sp(n, \Z)$, a change of 
basis is sometimes needed, 
replacing $G$ by a
conjugate in $\Sp(n, \Q)$; 
 cf.~\cite{Symplectica}.

Let $H$ be a finitely generated subgroup 
of $\GL(n, \Q)$. Then $H\leq \GL(n, R)$ 
where $R$ is a localization $\Z[1/\mu]$ 
of $\Z$ for some positive $\mu\in \Z$.
Often we take $R$ to be the subring $R_H$ 
of $\Q$ generated by the entries of $h$ 
and $h^{-1}$ as $h$ runs through a 
finite generating set for $H$.
Linear groups over $\Z$ can be 
quite different to those over $R$.
As one instance: if $\mu >1$ then 
 $\SL(2, R)$ has the congruence 
subgroup property (CSP), while 
$\SL(2, \Z)$  does 
not~\cite{Mennicke,SerreCSP,VenkyCSP}. 

For a commutative ring $D$ with $1$,
let $\Gamma(n, D)$ denote 
 $\SL(n, D)$ or $\Sp(n, D)$;
if $D$ is the finite field $\F_q$
 of size $q$, then we write $\Gamma(n,q)$.
Let $S= \{1/p_1, \ldots , \allowbreak 
 1/p_k\}$ where the $p_i$ are the 
distinct prime divisors of $\mu>1$. 
Then $R=\allowbreak \Z[S]$ 
and we call $H\leq \Gamma(n,R)$ 
an \emph{$S$-integral group}. 
If $H$ has finite index in 
$\Gamma(n,R)$, then it
 is \emph{$S$-arithmetic}.
Each non-zero ideal 
$I\subset R$ is principal:
 $I = mR$ for a unique positive 
 integer $m$ coprime to $\mu$.
The congruence homomorphism modulo 
$I$ (namely, the reduction modulo $m$ map) is 
denoted $\varphi_m$. Thus 
$\varphi_m$ maps $R$ onto $\Z_m:= \Z/m\Z$,
and by entrywise extension we get a 
homomorphism $\Gamma(n, R)\rightarrow 
 \Gamma(n, \Z_m)$, 
 denoted $\varphi_m$ as well. 
The kernel $\Gamma_{n,m}$  
of $\varphi_m$ on $\Gamma(n,R)$ 
 is the 
\emph{principal congruence subgroup} (PCS)
of \emph{level $m$}. Observe that 
$\Gamma_{n,a}\Gamma_{n,b} = \Gamma_{n,d}$
and $\Gamma_{n,a}\cap \Gamma_{n,b} = 
\Gamma_{n,l}$
where $d = \mathrm{gcd}(a,b)$, 
 $l = \mathrm{lcm}(a,b)$.
 A subgroup of $\Gamma(n,R)$
that contains a PCS is a  
\emph{congruence subgroup}.
The foregoing terminology and notation
may be used for coefficient rings  
other than $\Z[S]$; e.g., if $r$ and
$m$ are positive integers with 
$r$ dividing $m$, then $\varphi_r$ on $\Z_m$ 
is the reduction modulo $r$ map.
Throughout, $\F$ stands for an arbitrary 
field, and $1_m$, $0_m$ denote
the $m\times m$ identity and
zero matrices, respectively. 

Our main algorithms accept
subgroups $H\leq \Gamma(n,R)$ that are 
Zariski dense in $\Gamma(n,\C)$.
Algorithms to test density are 
discussed in \cite{Density}.
Another input item is 
the set $\Pi(H)$ of all primes $p\in \Z$
such that $\varphi_p(H) \neq \Gamma(n,p)$.
Since $H$ is dense 
in $\Gamma(n, \mathbb{R})$,  
the set $\Pi(H)$ is finite
 (by the Strong Approximation 
theorem~\cite[Window~9]{LubotzkySegal}).
See \cite{ExpMath,SAT,Density}
for algorithms to compute $\Pi(H)$.
 
Section~\ref{ComputingQuotients} 
focuses on computing the level of $H$.
We generalize
the notion of arithmetic closure
as in \cite{Density} to that of 
congruence closure $\cl(H)$, which 
exists for any dense 
$H\leq \Gamma(n, R)$. Then 
we show how to compute the 
(level of the) unique maximal PCS 
of $\Gamma(n, R)$ in $\cl(H)$, 
thereby determining $\cl(H)$.
Section~\ref{ComputingQuotients} also 
provides a thorough treatment of the 
case $n=2$, omitted 
 from \cite{Density}. 
We are then able to compute
congruence quotients of $H$
modulo all ideals of $R$ and 
for all $n\geq 2$.

Section~\ref{Exp} 
reports on experiments performed with 
our {\sf GAP}~\cite{Gap} implementation 
of the main algorithms. 
To emphasize the utility of
such experimentation, we include
the solution of a problem 
originating from further afield in
linear group theory.

\section{Computing the level of 
a dense $S$-integral group}

\label{ComputingQuotients}

Let $R=\Z[1/\mu]$ for a positive 
integer $\mu$.
If $n>2$, or $R\neq \Z$,
then  $\Gamma(n,R)$ has the CSP:
the $S$-arithmetic
subgroups of $\Gamma(n, R)$ 
are exactly its congruence subgroups.

In \cite[$\S \hspace{1pt} 3.3$]{Density}, 
we defined the \emph{arithmetic closure} 
of a dense group in $\Gamma(n, \Z)$, 
$n > 2$. As an intersection of
congruence subgroups, the arithmetic
closure is also called an 
\emph{extended congruence 
 subgroup}~\cite{VenkyExtended}.
We now broaden this concept to
\emph{congruence closure}
of any dense group 
in $\Gamma(n, \Q)$, $n\geq 2$.
Thereafter we devise algorithms to 
compute the closure.

If $n \geq 2$ then much of 
\cite[$\S \hspace{1pt} 2$]{Density}
carries over from $\Z$ to $R$
 without complication, simply by
avoiding the primes that divide 
$\mu$; see Section~\ref{case1}.
We deal with degree $2$ 
specifically
in Section~\ref{case2}.

\subsection{General degree}
\label{case1}

Let $H$ be a finitely generated 
dense subgroup of $\Gamma(n,\Q)$
for $n\geq 2$,
and let $R = R_H$.
\begin{theorem}
\label{FiniteNumberOvergroups} 
The intersection of all congruence 
subgroups of $\hspace{.5pt} \Gamma(n,R)$ 
containing $H$ is a congruence 
subgroup of $\hspace{.5pt} \Gamma(n,R)$.
\end{theorem}
\begin{proof}
See \cite[Theorem~2, p.~391]{LubotzkySegal}.
\end{proof}

Let $\cl(H)$ be the intersection of all 
congruence subgroups of $\Gamma(n, R)$ 
that contain $H$. By 
Theorem~\ref{FiniteNumberOvergroups},
$\cl(H)$ is the unique smallest such 
congruence subgroup:
the \emph{congruence closure 
of $H$} in $\Gamma(n,R)$. 
\begin{remark}
The definition of $\cl(H)$
is subject to the choice 
of coefficient ring $R$.  
For example, 
 $H=\SL(n, \Z)$ is a maximal 
subgroup of $\SL(n, \Z [1/p])$ for 
any prime $p$, so that 
$\cl(H)=\SL(n, \Z [1/p])$ over 
$R=\Z[1/p]$;  whereas 
of course  $\cl(H)=H$ over $R=\Z$.
\end{remark}

\begin{remark}
If $\Gamma(n,R)$ has the CSP, then
$\cl(H)$ is the same as the
\emph{$S$-arithmetic closure} of $H$,
i.e., $\cl(H)$ is
 the intersection of all
$S$-arithmetic subgroups of 
$\Gamma(n,R)$ that contain $H$.
Thus, our current definition of 
$\cl(H)$ agrees with the one
 in \cite{Density} 
for groups $H$ over $\Z$.
\end{remark}

\begin{lemma}[{cf.~\cite[Lemma~3.11]{Density}}]
\label{LellSame}
$\Pi(H) = \Pi(\cl(H))$.
\end{lemma}

As a congruence subgroup, $\cl (H)$ 
contains a largest PCS; let its level
be $\ell$. Hence 
$\cl(H)=\Gamma_{n,\ell}\hspace{.5pt} H$.
Furthermore, $\Gamma_{n,m}\leq \cl(H)$ 
if and only if 
$\ell \hspace{1.5pt} 
|\hspace{1.5pt} m$.
We say that  $H$ itself 
has level $\ell$.

\begin{lemma}\label{EasyWay}
Each prime in $\Pi(H)$ divides 
the level of $H$.
\end{lemma}
\begin{proof}
Let $\Gamma(n,\Q)=\SL(n,\Q)$; 
the proof for $\Gamma(n,\Q)=\Sp(n,\Q)$ 
is similar. 

Let $t_{ij}(a)$ be the  
elementary matrix with $1$s down 
the main diagonal, $a$ in off-diagonal
position $(i,j)$, and $0$s elsewhere.
Suppose that $p\in \Pi(H)$ does not 
divide the level $\ell$ of $H$. Then 
$\varphi_p(\Gamma_{n,\ell})$
 contains $t_{ij}(1)= \allowbreak 
 \varphi_p(t_{ij}(\ell))^{k}$
where $k\equiv \ell^{-1} 
 \allowbreak \bmod p$. 
Since the $t_{ij}(1)$
 generate $\SL(n,p)$, we get 
$\varphi_p(\cl(H)) = 
\varphi_p(H\Gamma_{n,\ell} ) = 
\allowbreak \SL(n,p)$.
Hence $p\not \in \Pi(H)$ by 
Lemma~\ref{LellSame}.  
\end{proof}

Define 
\[
\delta_H(m)=
|\Gamma(n,R): \Gamma_{n,m}H|.
\]
That is, $\delta_H(m)=
|\varphi_m(\Gamma(n,R)):\varphi_m(H)|$, 
and $\Pi(H)$ comprises the
primes $p$ not dividing $\mu$ 
with $\delta_H(p)>1$.
Additionally 
$\delta_H(m) = \delta_{\cl(H)}(m)$
 for all $m>1$.
 
Denote the set of prime 
divisors of an integer $k$ by $\pi(k)$.
For arithmetic $H\leq \Gamma(n,\Z)$, $n>2$, 
we proved in \cite{Density}
that $\pi(\ell)= \widetilde{\Pi}(H)$, 
where $\widetilde{\Pi}(H)\supseteq \Pi(H)$ 
is defined in \cite[(2.3)]{Density}. 
The calculation of $\widetilde{\Pi}(H)$
 incorporates a test on certain
$\delta_H$-values; see 
Theorem~\ref{TildePiSmallPi}.
We call $2$ an `exceptional' prime, 
because when $2\leq n\leq 4$, 
the arithmetic group $H$ could have 
even level yet surject onto
$\Gamma(n,2)$ modulo $2$. However,
$2$ is the single potential disparity between
$\widetilde{\Pi}(H)$ and $\Pi(H)$ when
$2<n\leq 4$.
As we will see, new exceptional primes
 occur when $n=2$. 

We refer to
\cite[$\S \hspace{1pt} 2.4.2$]{Density}
for an explanation of the procedure 
${\tt LevelMaxPCS}$ to compute the level 
$\ell$ of an arithmetic 
group $H\leq \Gamma(n,\Z)$, $n>2$. 
Apart from (a generating set
for) $H$, the other 
input is a set of primes:
in practice,  
 $\widetilde{\Pi}(H)$.
 Below, we paraphrase the pseudocode
 for $\tt LevelMaxPCS$ from
\cite[p.~975]{Density}.

\vspace{17.5pt}

${\tt LevelMaxPCS}(S, \Omega)$

\vspace{5pt}

{\sc Input}: a generating set $S$ for a 
subgroup $H\leq \Gamma(n,\Q)$; 
a set $\Omega$ of primes.

{\sc Output}: an integer $r$.

\vspace{5pt}

\begin{itemize}
\item[] For each $p\in\Omega$ let $\nu_p=1$ and 
$z_p=\prod_{r\in\Omega, r\neq p} r$.

\item[] \vspace{-10pt}

\item[] While $\exists \, p\in\Omega$ such that
$\delta_H(p^{\nu_p+1} \cdot z_p)>
 \delta_H(p^{\nu_p} \cdot z_p)$
\begin{itemize} 
\item[] increment $\nu_p$ by $1$ and repeat.
\end{itemize}

\item[] \vspace{-10pt}

\item[] Return  $r:=\prod_{p\in\Omega}p^{\nu_p}$.
\end{itemize}

\vspace{17.5pt}

\begin{theorem}\label{level} 
${\tt LevelMaxPCS}$ with 
input dense $H$ of level $\ell$ and 
$\Omega = \pi(\ell)$ terminates, 
returning $\ell$. 
\end{theorem}
\begin{proof}
This is almost identical to the 
proof of \cite[Theorem~2.22]{Density}.
\end{proof}

\begin{remark}
We obtain a generating set 
for $\cl(H)$ if we know a generating 
set for each PCS in $\Gamma(n,R)$;
cf.~\cite{SuryVenky}.
\end{remark}

\begin{remark}\label{ProfiniteNascent}
Suppose that $\Pi(H)=\emptyset$. 
By \cite[(2.3)]{Density}, if either 
$n>4$, or $3\leq n\leq 4$ and 
$\delta_H(4)=\allowbreak 1$,
then $\widetilde{\Pi}(H)=\emptyset$ too.
This means that $H$ has level $1$, 
i.e., $\cl(H) = \Gamma(n,R)$, and 
$\tt LevelMaxPCS$ with input $H$ returns $1$ 
by default.
For $H\leq \SL(n,\Z)$, in this 
situation $H$ is 
\emph{profinitely dense}:
$\delta_H(m)=1$ for all
$m >1$.
\end{remark}

The computation of $\pi(\ell)$ is 
a separate endeavor, undertaken 
before running $\tt LevelMax$-$\tt PCS$. 
 As noted, in practice 
$\pi(\ell)$ will be input as 
$\widetilde{\Pi}(H)$. 
The latter is defined for $n=2$ in 
Section~\ref{case2}. 
We confirm that $\widetilde{\Pi}(H)=\pi(\ell)$
 by combining Theorems~\ref{TildePiSmallPi} 
and \ref{analogue217}.
\begin{theorem}
\label{TildePiSmallPi} 
For $n>2$, let $H\leq \Gamma(n,\Q)$ be 
dense of level $\ell$.
Let $q$ be the product of 
the odd primes in $\Pi(H)$. 
Then $\ell$ is even if and only if 
$\delta_H(q)<\delta_H(4q)$.  
Therefore $\widetilde{\Pi}(H) = \pi(\ell)$.
\end{theorem}
\begin{proof}
It is straightforward to check that 
the proof of 
\cite[Theorem~2.17]{Density}, for 
arithmetic $H\leq \Gamma(n,\Z)$, carries
 over to the present setting.
\end{proof}
\begin{remark}\label{212Debug}
The proof of \cite[Theorem~2.17]{Density}
depends on \cite[Lemma~2.12]{Density}. We 
point out that the condition `$q>3$ odd' in 
this lemma is redundant (i.e., 
it should be changed to  `$q>1$').
\end{remark}

${\tt LevelMaxPCS}$ enables us 
to describe all congruence 
 quotients of $H\leq \Gamma(n,R)$, 
i.e., the congruence images 
$\varphi_m(H) = 
\varphi_m(\cl(H))$ for
$m$ coprime to $\mu$;
the procedure
$\tt PrimesFor$-$\tt Dense$~\cite[$\S 
\hspace{1pt} 3.2$]{Density}
(which accepts a transvection in 
 $H$ and a generating 
set for $H$)
 furnishes all congruence quotients 
 modulo maximal ideals of $R$.
To expand on the former comment,
we take $H$ to be a congruence subgroup,
of level $\ell$.
Let $a=\mathrm{gcd}(m,\ell)$, so
$m=abc$ where $\pi(b)\subseteq \pi(a)$ and
$\mathrm{gcd}(c,a ) = \mathrm{gcd}(c,\ell ) =1$.
Then $\varphi_m(H)$ is 
isomorphic to a subgroup of 
$\Gamma_{n,a}/\Gamma_{n,m}\times
\Gamma_{n,c}/\Gamma_{n,m}$.
Moreover, the first factor 
in this decomposition
 has a normal subgroup isomorphic 
to $\Gamma(n,\Z_c)$ that splits off
$\varphi_m(H)$.
Hence we may further assume that 
$\pi(m)\subseteq \pi(\ell)$. 
The congruence quotient
$H\Gamma_{n,a}/\Gamma_{n,m}$ is 
then an extension of the PCS  
of level $a$ in 
$\Gamma(n,\Z_m)$ by $\varphi_a(H)$.

Knowing the level of 
a congruence subgroup 
 $H\leq \Gamma(n, R)$, we 
can solve other problems 
for $H$, such as calculating 
$|\Gamma(n,R) :H|$,
and testing membership of 
$g \in \Gamma(n,R)$ in $H$
 (cf.~Grunewald and Segal's 
`explicitly given' arithmetic
$\Q$-groups~\cite[$\S 
 \hspace{.5pt} 3.1$]{Arithm}).

\subsection{Degree $2$}
\label{case2}

In this subsection, 
$R = \Z[1/\mu]$, $\mu$ a positive
integer, and $\Gamma=\SL(2,R)$.
For an integer $m>1$ coprime to
$\mu$, we denote the PCS $\ker \varphi_m$ 
of level $m$ in $\Gamma$  by $\Gamma_m$.

Let $H\leq \Gamma$ be dense
(and finitely generated). As always,
 $\cl(H)$ denotes the congruence 
closure of $H$ in $\Gamma$, with the 
standard assumption $R = R_H$.
The $S$-arithmetic closure of $H$ in 
$\Gamma$ is contained in $\cl(H)$, and
they can differ only when $R = \Z$.

Our major result for degree $2$ is 
the new specification of the input 
$\widetilde{\Pi}(H)$ to 
$\tt LevelMaxPCS$ 
(otherwise, the operation of this 
procedure remains entirely the same).
We isolate $2,3,5$ as 
\emph{exceptional} primes: 
$2$ and $3$ because $\SL(2,\Z_{p^2})$ 
for $p\in\{2,3\}$ has proper 
supplements to its PCS of level $p$,
and $2,3,5$ because there can be
non-trivial subdirect products of
$\SL(2,\Z_p)$ with subgroups of 
$\SL(2,\Z_r)$ for $p\in\{2,3,5\}$ 
and prime $r\neq p$.
(To reiterate: in degrees greater than
$2$, only $2$ can be an exceptional prime,
and then only in degree $3$ or $4$.)
\begin{lemma}
\label{analogue29}
Let $G=\SL(2,\Z_{p^a})$, $p$ prime.
\begin{itemize}
\item[{\rm (i)}] If $p>3$ then each
proper normal subgroup 
of $G$ lies in the solvable radical 
of $G$, which is the complete inverse 
image of $Z(\SL(2,p))$
under $\varphi_p$.
\item[{\rm (ii)}]  If $p=2$ or $3$ then
$G$ is solvable, and again 
has a unique maximal normal subgroup.
\end{itemize}
\end{lemma}
\begin{proof}
(i) \hspace{-.5pt} Cf.~the proof of 
 \cite[Lemma~2.9]{Density}.

(ii) \hspace{-.5pt}
 This can be checked by direct 
calculation for
$p^a\in\{2,3,4,9\}$. 
For larger $a$, 
any other maximal normal subgroup 
would be a proper supplement to 
the PCS modulo $p^{a-1}$ in $G$,
violating \cite[Theorem~2.5]{Density}.
\end{proof}

\begin{lemma}
\label{analogue2.11}
If $q>5$ is a prime power 
and $p\neq q$ is a  prime, then 
$\SL(2,p)$ does not have $\PSL(2,q)$
as a section.
\end{lemma}
\begin{proof}
(Cf.~\cite[Lemma 2.11]{Density}.)
The assertion is evident from Dickson's 
list of the subgroups of 
$\PSL(2,p)$~\cite[Hauptsatz~8.27, p.~213]{Huppert}, 
and simplicity of $\PSL(2,q)$.
\end{proof}

We shall use Lemmas~\ref{analogue29} 
and \ref{analogue2.11} in place of
\cite[Lemmas~2.9, 2.11]{Density} to
prove degree $2$ versions of
\cite[Lemma~2.12]{Density} (amended;
see Remark~\ref{212Debug})
 and \cite[Theorem~2.13]{Density}. 
\begin{lemma}[cf.~{\cite[Lemma~2.12]{Density}}]
\label{NewKernelcor}
Let $a$, $k$ be positive integers,
and let $p>5$ be a prime not dividing $k$.
Suppose that $U\le\SL(2,\Z_{p^ak})$
maps onto $\SL(2,\Z_{p^a})$ under  
the projection $\varphi_{p^a}$
of $\SL(2,\Z_{p^ak})\cong 
\SL(2,\Z_{p^a})\times\SL(2,\Z_k)$
onto its first factor.
Then $U$ contains $\SL(2,\Z_{p^a})$.
\end{lemma}
\begin{proof}
By hypothesis, $U$ is a subdirect product 
of $A=\SL(2,\Z_{p^a})$ with 
$B = \allowbreak \varphi_k(U)$, i.e.,
$U \leq A\times B$ and the 
component projections of $AB$ onto $A$ and
$B$ restricted to $U$ are both surjective. 
We prove that $B\leq U$; then  
$A\leq U = A\times B$.

Suppose that $k=r^b$ for a prime $r$.
If $B\not \leq U$ then $A$ and $B$ have 
isomorphic non-trivial quotients. 
By Lemma~\ref{analogue29}, any 
non-trivial quotient of $A$ has 
a quotient isomorphic to $\PSL(2,p)$; 
so $r>3$ and $\PSL(2,p)$ is a
section of 
$\varphi_r(\SL(2,\Z_k)) = \SL(2,\Z_r)$.
But this contradicts 
Lemma~\ref{analogue2.11}.

Suppose now that $k=r^b s$ 
where $r$ is a prime $r$
not dividing $s$. 
Then $\varphi_{p^ar^b}(U)=
A \times \varphi_{r^b}(U)$ 
by the previous paragraph.
Therefore $\varphi_{p^a}$
maps $W:=U\cap \hspace{.5pt}
\SL(2,\Z_{p^as})$ onto $A$. 
We replace $U$ by $W$, and $k$ by 
$s$, then recurse.
After exhausting all prime 
divisors of $k$, we reach 
a subgroup of $U$ that contains
$A$. 
\end{proof}

\begin{theorem}[cf.~{\cite[Theorem~2.13]{Density}}]
\label{CarryOver213}
Let $H\leq \Gamma$ be dense, of
level $\ell >1$.
Suppose that $\varphi_p(H)=\SL(2,p)$ 
for some prime $p>5$. Then $p\nmid \ell$.
\end{theorem}
\begin{proof}
Assume that $\varphi_{p^{k-1}}(H)=
 \SL(2,\Z_{p^{k-1}})$ for some $k\geq 2$, 
i.e., $\varphi_{p^k}(H)$ supplements 
$\ker \varphi_{p^{k-1}}$ in
$\SL(2,\Z_{p^k})$. Then
$\varphi_{p^k}(H)=\SL(2,\Z_{p^k})$
by \cite[Theorem~2.5]{Density}.
This proves that 
 $\varphi_{p^k}(H)=\SL(2,\Z_{p^k})$
 for all $k\geq 1$, by induction.

Let $\ell =p^a q$ where 
$\mathrm{gcd}(p, q)=1$.
By Lemma~\ref{NewKernelcor},
$\Gamma_{q}\leq  H\hspace{.3pt} 
\Gamma_\ell = \cl(H)$.
Since $\ell$ divides the level of 
every PCS in $\cl(H)$,
necessarily $a=0$.
\end{proof}

By Theorem~\ref{CarryOver213},
 $\pi(\ell)\setminus \{2,3,5\}
\subseteq \Pi(H)$. We knew already
(Lemma~\ref{EasyWay}) that 
$\Pi(H)\subseteq \pi(\ell)$.
Now we seek the exceptional primes
in $\pi(\ell)$.
Toward this end, define $r_q= r_q(H)$ 
for a prime $q$ to be the product of
 all primes greater than $q$ that 
divide $\ell$; when no such primes 
exist, $r_q=1$. 

\begin{theorem}[cf.{~\cite[Theorem~2.17]{Density}}]
\label{analogue217}
Let $H\leq \Gamma$ be dense, of level $\ell>1$. 
Suppose that $p\not \in \Pi(H)$, where
$p$ is one of $\, 2,3$, or $5$.
Then $p\in \pi(\ell)$ if and only if 
\begin{itemize}
\item[{\rm (i)}] 
$\delta_H(r_5)<\delta_H(5r_5)$, for $p=5;$
\item[{\rm (ii)}] 
$\delta_H(4)<\delta_H(36)$ or 
$\delta_H(2r_3)<\delta_H(6r_3)$, 
 for $p=3;$
\item[{\rm (iii)}] 
$\delta_H(9)<\delta_H(36)$ or 
$\delta_H(3r_3)<\delta_H(6r_3)$, 
 for $p=2$.
\end{itemize}
\end{theorem}

We use Theorem~\ref{analogue217} 
to find the set $E$ of
all $p\in \{2,3,5\}$ such that
$\delta_H(p)=1$ and 
$p\hspace{1.5pt} | 
\hspace{1.5pt}\ell$.
Then $\pi(\ell) = 
\widetilde{\Pi}(H):= 
\Pi(H)\cup \allowbreak E$ 
 (cf.~\cite[(2.3)]{Density}). 

One more lemma is required for the proof 
of Theorem~\ref{analogue217}.
 In the sequel, `psp' is an 
abbreviation of `proper subdirect 
 product'. 
\begin{lemma}\label{sudilem}
Let $A$ and $B$ be groups, 
and let $U \lneq A\times B$ be a 
psp of $A$ with $B$. 
Suppose that $A/(U\cap A)$ has a unique
maximal normal subgroup $M/(U\cap A)$.
Then $B/(U\cap B)$ has a unique maximal 
normal subgroup $N/(U\cap B)$ 
such that $UMN/MN$ is a psp of $A/M$ 
with $B/N$.
\end{lemma}
\begin{proof}
There is an isomorphism 
$\theta \colon A/(U\cap A)\rightarrow 
B/(U\cap B)$ such that
\[
U = \{ ab \; |\;  a\in A, \, b\in B, \, 
\theta (a(U\cap A)) =
b(U\cap B)\}.
\]
For $X=A$ or $B$, let 
$\rho_X\colon A\times B \rightarrow X$
be the canonical projection map. By
assumption, $\rho_X(U) = X$.

Define $N\leq B$
by $N/(U\cap B)= \theta(M/(U\cap A))$.
Set $K=MN$.
Clearly $UK/K$ is a subdirect 
product of $A/M$ with $B/N$.

Suppose that $UK = AB$. 
Let $a\in A$, so $a = uxy$ for 
 $u\in U$, $x\in M$, and 
$y\in \allowbreak N$. 
Then $1 = \rho_B(u)y$, which implies
that $\rho_B(u) \in N$.
Therefore $\rho_A(u) \in M$. 
However, $a = \allowbreak \rho_A(u)x$,
yielding the contradiction 
$A\subseteq M$. 
\end{proof}

\noindent \emph{Proof of 
Theorem{\em ~\ref{analogue217}}}.
We fix some notation: $\delta=\delta_H$,
$t = \mathrm{gcd}(6,\ell)$, and
$s$ is the $\{2,3\}$-part of $\ell$.

For each $p\in \{2,3,5\}$,
it is not difficult to prove that 
$p\hspace{1.25pt} | \hspace{1.25pt} \ell$
if the stated condition on 
$\delta$-values holds.
We proceed to establish the converse.

Let $\ell=5^e s q$ for 
$e\ge 1$ and $q$ coprime to $30$. 
By \cite[Theorem~2.5]{Density}
and the definition of level,
$\varphi_\ell(H)$ is a psp of 
$\varphi_{5^e}(H) = \SL(2,\Z_{5^e})$ 
with $\varphi_{qs}(H)$.
Recall from Lemma~\ref{analogue29} 
that $\SL(2,\Z_{r})$ for any
prime-power $r$ has a unique maximal 
normal subgroup. 
By Lemma~\ref{sudilem}, then, 
$\varphi_{5r_5t}(H)$ is a psp of 
$\SL(2,5)$ with $\varphi_{r_5 t}(H)$. 
Since $\SL(2,\Z_{t})$ does not have 
$\PSL(2,5)$ as a section, 
 $\varphi_{5r_5}(H)$ must be a psp 
 of $\SL(2,5)$ with $\varphi_{r_5}(H)$.
Thus $\delta(r_5)\neq \delta(5r_5)$. 

Now let $p\in \{2,3\}$ be a divisor 
of $\ell$. 
If $\delta(4)=\delta(36)$
and $\delta(9)=\delta(36)$,
then $\delta(36) = \allowbreak 1$;
in turn, by \cite[Theorem~2.5]{Density},
 $\varphi_{m}(H) = \allowbreak \SL(2,\Z_{m})$
for all $\{2,3\}$-numbers $m$.
Consequently $\varphi_\ell(H)$ is a 
psp of $\SL(2,\Z_s)$ with 
$\varphi_{\ell/s}(H)$.
Reasoning as in the $p=5$ case (so 
appealing to 
Lemma~\ref{sudilem} again), we infer 
that $\varphi_{tr_3}(H)$ is a psp of 
$\SL(2,p)$ with $\varphi_{r_3t/p}(H)$.
The inequalities
$\delta(6r_3)> \delta(3r_3)$ for $p=2$ 
and $\delta(6r_3)>\delta(2r_3)$ for 
$p=3$ follow. \hfill $\Box$

\vspace{10pt}

We illustrate Theorem~\ref{analogue217} 
with a few examples in $\Gamma = \SL(2,\Z)$.
The groups $G_i$ below are dense by 
\cite[Proposition~3.7]{SAT}.

Let
\[
G_1=\Big\langle \hspace{1pt}
{\small\begin{bmatrix}
946549&-1531420\\
-28215&45649\\
\end{bmatrix}}, \hspace{1pt}
{\small\begin{bmatrix}
281823&-631463\\
181643&-406996\\
\end{bmatrix}}, \hspace{1pt}
{\small\begin{bmatrix}
47149&1057765\\
32466&728359\\
\end{bmatrix}} \hspace{1pt}
\Big\rangle .
\]
 $\tt PrimesForDense$ 
(see \cite[$\S \hspace{1pt} 3.2$]{Density})
returns $\Pi(G_1)=\{11\}$, i.e., $r_5 = 11$.
We compute that 
$\delta_{G_1}(11)=11$ 
and $\delta_{G_1}(36)=1$, so
$\delta_{G_1}(4)=\delta_{G_1}(9)=1$.
Since $\delta_{G_1}(55)=1320$, 
Theorem~\ref{analogue217}~(i) tells
us that $5\in \widetilde{\Pi}(G_1)$. 
Thus $r_3=55$. Since
$\delta_{G_1}(2.55)=\delta_{G_1}(3.55)=
\delta_{G_1}(6.55)=1320$, 
we have $\widetilde{\Pi}(G_1) = \{ 5,11\}$
by Theorem~\ref{analogue217}~(ii), (iii).

Next, let
\[
G_2=\Big\langle \hspace{1pt}
{\small\begin{bmatrix}
-11&-105\\
\phantom{-}2&\phantom{-}19\\
\end{bmatrix}}, \hspace{1pt}
{\small\begin{bmatrix}
\ 9&161\\
\ 1&18\\
\end{bmatrix}}, \hspace{1pt}
{\small\begin{bmatrix}
-4&-17\\
\phantom{-}1&\phantom{-}4\\
\end{bmatrix}}  \hspace{1pt}
\Big\rangle .
\]
Then $\Pi(G_2)=\emptyset$: 
the level $\ell$ of $G_2$ is
a $\{2,3,5\}$-number. 
Thus $r_5=1$. Since
$\delta_{G_2}(5)=1$, 
$\ell$ is a $\{2,3\}$-number 
by Theorem~\ref{analogue217}~(i).
We find that  
$\delta_{G_2}(4)=\allowbreak 4$, 
$\delta_{G_2}(6)=\allowbreak 1$,
$\delta_{G_2}(9)=\allowbreak 27$,
and $\delta_{G_2}(36)=108$.
 Hence
$\widetilde{\Pi}(G_2)=\{ 2,3\}$ by
(ii) and (iii) of Theorem~\ref{analogue217}.
(Note that $r_3=\allowbreak 1$, so
$\delta_{G_2}(2r_3)=\delta_{G_2}(3r_3)=
\delta_{G_2}(6r_3)=1$.)

For a third example, let
\begin{eqnarray*}
&G_3=\Big\langle \hspace{1pt}
{\scriptsize
\begin{bmatrix}
-23653189915040069676&78121325235187312055\\
-9796772300912447255&32356601283956634324 
\end{bmatrix}}, \hspace{1pt}
& \\
&{\scriptsize
\begin{bmatrix}
-1557884485479161873562328016576807435&3819960293513731249430174164790804354
\\
-1900057682869268441725410979144592304&4658975021318032809514788603728789629
\end{bmatrix}}, \hspace{1pt}
&\\
&
{\scriptsize 
\begin{bmatrix}
-389&-285\\
\phantom{-}1455& \phantom{-}1066 
\end{bmatrix}}, \hspace{1pt}
{\scriptsize 
\begin{bmatrix}
166&-225\\
-45&\phantom{-}61 
\end{bmatrix}} \hspace{1pt}
\Big\rangle&\\
\end{eqnarray*}

\vspace{-4mm}

\noindent 
(the huge entries come from 
normalizing determinants to $1$).
Then $\Pi(G_3)=\{5\}$: so $r_3=5$
and $r_5=1$. 
We compute that
$\delta_{G_3}(5)=20$, $\delta_{G_3}(10)=40$, 
$\delta_{G_3}(15)=\allowbreak 60$, 
$\delta_{G_2}(30)=\allowbreak 120$, and  
$\delta_{G_3}(36)=1$.
Therefore $\widetilde{\Pi}(G_3)=\{ 2,3,5\}$.

Using the {\sf GAP} package 
\mbox{kbmag}~\cite{kbmag}, 
it may be verified that $G_1$ and
$G_2$ are \emph{thin} (i.e., they have 
infinite index in $\Gamma$).
We could not decide thinness of $G_3$ 
likewise, as the computation fails to 
terminate.

We further compute that 
$|\Gamma: \cl(G_i)|=
1320, 108, 120$ for 
$i = 1,2,3$, respectively.

\subsection{Testing whether a subgroup 
is a congruence subgroup}

Let $H\leq \Gamma= \SL(2,\Z)$ be 
arithmetic (thus dense).
In \cite[p.~147]{Newman},
the `level' of $H$ is defined 
to be the least positive $k\in \Z$
such that $H$ contains the 
$\Gamma$-normal closure of 
$\big \langle 
{\tiny
\begin{bmatrix}
1 & k \\
0 & 1
\end{bmatrix}} 
\big\rangle$. 
\begin{proposition}
Let $H$ have level $\ell$
(in our usual sense) and 
`level' $k$. 
Then $H$ is a congruence subgroup 
of $\hspace{.5pt} \Gamma$ 
if and only if $\Gamma_k\leq H$. 
In that event, $k=\ell$.
\end{proposition}
\begin{proof}
For the first claim,   
see \cite[Theorem~VIII.8]{Newman}. 
The second is then easily derived
from the two definitions of level. 
\end{proof}

Relying on \cite[Theorem VIII.8]{Newman}, 
Hsu~\cite{Hsu} designed an algorithm to
test whether $H$
is a congruence subgroup of $\Gamma$. 
The algorithm is 
applied in \cite[Example~2.5]{Hsu},
proving that a certain subgroup
 of index $10$ in $\Gamma$ 
is not a congruence subgroup.
Indeed, running ${\tt LevelMaxPCS}$ 
on this subgroup shows that it
has level $1$.

Alternatives to the approach in
\cite{Hsu} are possible.
If we have a membership test and
a generating set for each PCS 
in $\Gamma$, 
then we can decide whether $H = \cl(H)$. 
Or we can inspect finite 
quotients of $H$. If $H$ has a 
simple finite quotient 
not isomorphic to a subgroup
of any $\SL(2,p)$, then $H$ is not 
a congruence subgroup (in 
\cite[$\S \hspace{1pt} 4.2.4$]{ExpMath}, 
we used this property to justify 
thinness of subgroups of $\SL(3, \Z)$). 

\section{Experimentation}
\label{Exp}

We discuss experiments conducted 
with a {\sf GAP} implementation of 
the main algorithms.
Examples were chosen to 
demonstrate practicality of the 
algorithms. Our {\sf GAP} code is 
available at 
\url{https://github.com/hulpke/arithmetic}.

Here we limited the selection of
parameters in test groups. Longer tables are posted at 
\url{https://www.math.colostate.edu/~hulpke/examples/Sintegral.pdf}.

\subsection{Subgroups of $\SL(2, \Q)$}
\label{sl2Z}

In \cite{Moebius}, we 
investigated the classical
problem of certifying non-freeness 
of (parabolic, $2$-generated) 
M{\"o}bius groups.
In this subsection we continue
experimentation with 
subgroups of $\SL(2, \Q)$. 

For $s\in \Q$, let
$G_s = \langle T, K_s \rangle$ where
\[
T = 
	{\scriptsize
	\begin{bmatrix}
	\phantom{-}0 & 1\, \\
	-1 & 0 \,
	\end{bmatrix}}, \qquad
	K_s = 
	{\scriptsize
	\begin{bmatrix}
	0 & -s \, \\
	\frac{1}{s} &\phantom{-}1
	\end{bmatrix}} .
	\]
According to \cite[Theorem~VIII.16]{Newman},
$G_s =  \langle T\rangle 
\ast  \langle K_s \rangle 
\cong \allowbreak C_2\ast C_3 
\cong \mathrm{PSL}(2, \Z)$ is a 
discrete subgroup of $\SL(2,R)$.
If $s\neq 1$ then $G_s$ is 
thin in $\SL(2,R)$. 
Since $G_s$, $G_t$ are 
$\GL(2,\R)$-conjugate only
if $s= \allowbreak t$, the $G_s$ are
 pairwise isomorphic
non-conjugate
subgroups of $\SL(2,R)$.

We ran $\tt LevelMaxPCS$  
to compute the level $\ell_s$ of 
various $G_s$. Then we computed the index
$|\SL(2,\Z_{\ell_s}):\varphi_{\ell_s}(G_s)|$ 
of $\cl(G_s)$ in $\SL(2,R)$;
see Table~\ref{Ttopone}.
Despite the large indices, 
runtimes were negligible.

\begin{table}[h]
\begin{center}
\caption{Level and index 
of groups $G_s\leq \SL(2,\Q)$}
\label{Ttopone}
\begin{tabular}{l|l|l}
Group
&Level
&Index
\\
\hline
$G_{11}$
&$2^{2}3
{\cdot}13
{\cdot}1117$
&$2^{6}3^{3}7
{\cdot}13
{\cdot}31
{\cdot}43$
\\
$G_{11/3}$
&$2^{2}5^{2}13633$
&$2^{10}3
{\cdot}5^{2}17
{\cdot}71
{\cdot}401$
\\
$G_{11/5}$
&$2^{2}3
{\cdot}7
{\cdot}13
{\cdot}17
{\cdot}12241$
&$2^{16}3^{6}5
{\cdot}7^{2}13
{\cdot}17^{2}6121$
\\
$G_{11/7}$
&$2^{2}3
{\cdot}5^{2}13
{\cdot}17^{2}19
{\cdot}11113$
&$2^{18}3^{7}5^{3}7
{\cdot}13
{\cdot}17^{2}19
{\cdot}463
{\cdot}5557$
\\
\end{tabular}
\end{center}
\end{table}
\subsection{Experiments in 
degrees greater than $2$}
\label{furtherexp}

For $t,s,r\in \C$, let
\[
a(t) = 
{\scriptsize
\begin{bmatrix} 
1 & t & \tfrac{1}{2}t^2 & \tfrac{t^3}{6}\\
0 & 1 & t & \tfrac{1}{2} t^2 \\
0 & 0 & 1 & t \\ 
0 & 0 & 0 & 1\end{bmatrix}}, \quad
b(s) = 
{\scriptsize
\begin{bmatrix} 
1 & 0 & 0 & 0\\
0 & 1 & 0 & 0 \\
0 & 0 & 1 & 0 \\ 
s & 0 & 0 & 1
\end{bmatrix}}, \quad
c(r) = 
{\scriptsize \begin{bmatrix} 
1 & 0 & 0 & 0 \\ 
8r & 1 & 0 & 0 \\ 
48r^2 & 12r & 1 & 0 \\
224r^3 & 84r^2 & 14r & 1 
\end{bmatrix}}.
\]
Then $H_4(t, s) = 
\langle a(t), b(s) 
\rangle\leq \Sp(4, \C)$ 
and $H_4(t, r) = 
\langle a(t), c(r) \rangle
\leq \SL(4,\C)$ are
dense.
Furthermore, if $|t|>8$, $|s|>2$, and $|r|>1$, 
then $H_4(t,s)$ and $H_4(t,r)$ 
are free (see Proposition~3.8
and Example~3.10 of \cite{deGraaf}).

\subsubsection{$\Sp(4, R)$}
\label{sp4}	

Notice that $H_4(t, s)\leq \Sp(\Phi,R)$,
identifying the form $\Phi$ with its 
 matrix
\[
{\scriptsize
\begin{bmatrix} 
\phantom{-}0 & 0 & \! \! \phantom{-}0 & 1\, \\
\phantom{-}0 & 0 & \! \!        -1 & 0 \, \\
\phantom{-}0 & 1 & \! \!\phantom{-}0 & 0 \, \\ 
          -1 & 0 & \! \!\phantom{-}0 & 0 \, 
\end{bmatrix}}.
\]
Since $b(s)$ is a transvection, we can compute
  $\Pi(H_4(t, s))$ via ${\tt PrimesForDense}$.
Table~\ref{Ttoptwo}
displays the level of $H_4(t, s)$ 
and the index of its congruence closure 
in the ambient linear algebraic group 
$\Sp(4, R)$. 

\begin{table}[h]
\begin{center}
\caption{Level and index 
of groups $H_4(t, s)
\leq \Sp(4, R)$}
\label{Ttoptwo}
\begin{tabular}
{l|l|l|l}
$t$&$s$&Level&Index\\
\hline
$25/3$
&$29/3$
&$5^{13}29^{2}$
&$2^{14}3^{4}5^{65}7^{2}13
{\cdot}29^{5}421$
\\
$65/7$
&$3$
&$5^{7}13^{6}$
&$2^{14}3^{4}5^{31}7^{2}13^{28}17$
\\
$71/7$
&$17$
&$17^{2}71^{6}$
&$2^{20}3^{8}5^{3}7^{2}17^{5}29
{\cdot}71^{27}2521$
\\
$21/2$
&$5$
&$3^{6}5^{2}7^{6}$
&$2^{23}3^{29}5^{11}7^{27}13$
\\
$66/5$
&$26/7$
&$2^{9}3^{6}11^{6}13^{2}$
&$2^{61}3^{31}5^{5}7^{2}11^{27}13^{5}17
{\cdot}61$
\\
$27/2$
&$29/3$
&$29^{2}$
&$2^{7}3^{2}5^{2}7^{2}29^{5}421$
\\
$95/7$
&$3$
&$5^{7}19^{6}$
&$2^{14}3^{6}5^{32}13
{\cdot}19^{27}181$
\\
$97/7$
&$29/7$
&$29^{2}97^{6}$
&$2^{20}3^{4}5^{3}7^{6}29^{5}97^{27}421
{\cdot}941$
\\
$17$
&$13$
&$13^{2}17^{6}$
&$2^{18}3^{6}5^{2}7^{2}13^{5}17^{28}29$
\\
$26$
&$23/3$
&$2^{7}13^{6}23^{2}$
&$2^{44}3^{6}5^{3}7^{2}11^{2}13^{27}17
{\cdot}23^{5}53$
\\
$100$
&$16/3$
&$2^{19}5^{13}$
&$2^{113}3^{4}5^{64}13$
\\
\end{tabular}
\end{center}
\end{table}

\subsubsection{$\SL(4, R)$}
\label{sl4}
	 
We compute the level of 
$H_4(t, r)$
and then $|\SL(4,R):\cl(H_4(t,r))|$.
A sample of this data is compiled 
in Table~\ref{Ttopthree}.
Runtimes were negligible.
However, computing $\Pi(H_4(t,r))$
 is slow, as we do this via the algorithm 
${\tt PrimesNonSurjectiveSL}$
from \cite[$\S \hspace{1pt} 4$]{SAT}. 
This tests irreducibility of an adjoint 
representation.

\begin{table}[h]
\begin{center}
\caption{Level and index of 
groups $H_4(t,r)\leq \SL(4,R)$}
\label{Ttopthree}
\begin{tabular}{l|l|l|l}
$t$&$r$
&Level
&Index
\\
\hline
$10$ & $1/3$
&$2^{10}5^{4}7$
&$2^{87}3^{5}5^{24}7
{\cdot}13
{\cdot}31$
\\
$9$& $5/2$
&$3^{9}5^{4}7$
&$2^{22}3^{64}5^{24}13^{2}31$
\\
$10$& $4/3$
&$2^{19}5^{4}7$
&$2^{149}3^{5}5^{24}7
{\cdot}13
{\cdot}31$
\\
$15$& $7/4$
&$3^{5}5^{4}7^{5}$
&$2^{27}3^{37}5^{24}7^{25}13^{2}19
{\cdot}31$
\\
\end{tabular}
\end{center}
\end{table}

\subsubsection{$\SL(5, R)$}
Finally, we consider discrete faithful 
representations  
of the triangle group
$\Delta(3,3,4)$ in $\SL(5, \mathbb{R})$ 
(see \cite[$\S \hspace{1pt} 3$]{Long}
and 
\cite[$\S \hspace{1pt} 4.2.2$]{ExpMath}).
Let $H_3(k) = \langle a_3(k), 
\allowbreak b_3(k) \rangle$, where
\[
a_3(k) = 
{\small
\begin{bmatrix}
1 & 0 & -3-2k-8k^2 & -1+10k+32k^3 & -5-16k^2\\
0 & 4(-1+k) & -13 - 4k & 3+16(1+k)^2 & -4+16k \\
0 & 1-k+4k^2 & 3-2k+8k^2 & -2(1+3k+16k^3) & 3+16k^2\\
0 &  k & 2k & 1-2k-8k^2 & 1+4k\\ 
0 & 0 &  3k & 3(-1+k-4k^2) & -2
\end{bmatrix}}
\]
and
\[
b_3(k) = 
{\small 
\begin{bmatrix} 
0 & 0 & -3-2k-8k^2 & -1+10k+32k^3 & -5-16k^2\\
0 &1 & 3+4k & -13-8k-16k^2 & 4-16k\\
0 & 0 & -2(1+k+4k^2) & 6k+32k^3 & -3-16k^2\\
1 & 0 & -2(1+k) & -1+2k+8k^2 & -1-4k\\
2k & 0 & 1-2k &-4k & 1\end{bmatrix}}.
\]
\noindent 
If $k\in \Q$ then $H_3(k)$ is a thin 
subgroup of $\SL(5,R)$.
As usual, we compute 
the level and index for the 
congruence closure of several $H_3(k)$: 
 see Table~\ref{Ttopfour}.

\begin{table}[h]
\begin{center}
\caption{Level and index of 
groups $H_3(k)\leq \SL(5,R)$}
\label{Ttopfour}
{\tiny
\begin{tabular}{l|l|l}
\text{$k$}
&\text{Level}
&\text{Index}
\\
\hline
\vspace{-7.5pt} & & \\
$1/3$
&$2^{6}73^{2}$
&$2^{93}3^{10}5^{2}7
{\cdot}13
{\cdot}31
{\cdot}37^{2}41
{\cdot}73^{20}1801
{\cdot}28792661$
\\
$2/5$
&$2^{3}23^{2}29^{2}37$
&$2^{46}3^{11}5^{4}7^{6}11^{5}13
{\cdot}19
{\cdot}23^{20}29^{20}31
{\cdot}37^{4}41
{\cdot}53
{\cdot}67^{2}79
{\cdot}421
{\cdot}4271
{\cdot}292561
{\cdot}732541$
\\
$5/8$
&$139
{\cdot}151^{2}$
&$2^{14}3^{8}5^{10}7^{2}13^{2}19^{2}23^{3}41
{\cdot}139^{4}151^{20}499
{\cdot}877
{\cdot}1093
{\cdot}9170881
{\cdot}104670301$
\\
$4/5$
&$2^{3}11^{2}73
{\cdot}113^{2}$
&$2^{60}3^{9}5^{8}7^{5}11^{21}13^{2}19^{3}31
{\cdot}41
{\cdot}61
{\cdot}73^{4}113^{20}251
{\cdot}991
{\cdot}1277
{\cdot}1801
{\cdot}3221
{\cdot}59581
{\cdot}28792661$
\\
$7/6$
&$11^{2}23^{2}61$
&$2^{25}3^{7}5^{10}7^{2}11^{24}13
{\cdot}19
{\cdot}23^{20}31
{\cdot}53
{\cdot}61^{5}79
{\cdot}97
{\cdot}131
{\cdot}3221
{\cdot}21491
{\cdot}292561$
\\
$11/5$
&$2^{6}61^{2}97
{\cdot}103^{2}$
&$2^{112}3^{14}5^{8}7
{\cdot}11^{2}13^{3}17^{4}31^{4}61^{20}97^{5}103^{20}131
{\cdot}941
{\cdot}1061
{\cdot}1861
{\cdot}3169
{\cdot}3571
{\cdot}21491
{\cdot}262321
{\cdot}10332211$
\end{tabular}}
\end{center}
\end{table}

\noindent
We compute $\Pi(H_3(k))$ using 
procedures from \cite{ExpMath}
that do not entail adjoint 
representations.
Runtimes are much
shorter than those for
${\tt PrimesNonSurjectiveSL}$.

\subsection{An application}

G.~Soifer~\cite[Question~2.26]{BanffReport} 
asked:
\begin{quote}
\emph{Does there exist a profinitely dense 
subgroup of 
$\hspace{1pt} \SL(n, \Z)$ generated 
by two elements?}
\end{quote}

\noindent 
We answer this question affirmatively.

Certain free subgroups 
of $\SL(n, \Z)$ for $n \geq 3$  
are given by Humphries 
 in \cite{Humphries}.
Each surjects onto 
$\SL(n,p)$ modulo  all primes $p$.
One such group ($G_4$ 
in \cite[p.~983]{Density}) is 
generated by 
\[
f_1 = 
{\scriptsize \begin{bmatrix} 
1 & x^2+1 & x\\
0 & 1 & 0\\
0 & 0 & 1 \\
\end{bmatrix}}, 
\quad 
f_2 = {\scriptsize
\begin{bmatrix}
 1 & 0 & 0\\
x & 1 & x+1\\
0 & 0 & 1 \\
\end{bmatrix}},
\quad
f_3 ={\scriptsize 
\begin{bmatrix} 
1 & 0 & 0 \\
0 & 1 & 0\\
-x+1 & x^2 & 1 \\
\end{bmatrix}}
\]
for $x=11$. Let
\[
a=f_2 f_1 f_2^{-1} f_1^{-1} f_2 f_1^{-1} f_3^{-1} ,
\qquad 
 b=(f_1^{-1} f_3^{-2} f_2 f_3^{-1})^{f_2}.
\]
Explicitly,
\[
a=
{\scriptsize \begin{bmatrix}
-176379047&1936571625&-17817867\\
-1941616205&21318171799&-196142682\\
10&-121&1\\
\end{bmatrix}}, \ \ \
b=
{\scriptsize
\begin{bmatrix}
-40193888&-3684209 & -44180039\\
\phantom{-}396162230 &
\phantom{-}36312596 & \phantom{-}435450841\\
\phantom{-}3829559 & \phantom{-}351021 &
\phantom{-}4209349
\end{bmatrix}}.
\]
We verified that $H=\langle a,b\rangle$ 
 surjects onto $\SL(3, \Z_q)$ 
modulo $q=4$ and all primes $q$.
Thus $\varphi_m(H) =\SL(n, \Z_m)$ 
for all $m>1$, 
by Remark~\ref{ProfiniteNascent}.

Examples for $n>3$ that
answer Soifer's question affirmatively
 may be constructed in similar fashion 
(cf.~\cite[Corollary~3.3]{SoiferVenky}).

\subsection*{Acknowledgements}
We thank Centre 
International de Rencontres 
Math\'{e}matiques,
The Institute 
for Computational and Experimental 
Research in Mathematics, and
Mathematisches Forschungsinstitut 
Oberwolfach,  for  
support received through the
Research in Pairs, 
Collaborate@ICERM, and
Research Fellowship programmes.  
We thank Willem de Graaf for his
 assistance.
The third author's work has been 
supported in part by
NSF Grant DMS-1720146 and 
Simons Foundation Grant~852063, 
which are both gratefully 
acknowledged.

\bibliographystyle{amsplain}

\end{document}